\pdfoutput=1
\documentclass[11pt,a4paper]{amsart}
\usepackage[margin=1in]{geometry}
\usepackage[utf8]{inputenc}
\usepackage[T1]{fontenc}
\usepackage{textcase}
\usepackage[dvipsnames]{xcolor}
\usepackage{microtype}
\usepackage{fnpct}

\usepackage{amsmath}
\usepackage{amssymb}
\usepackage{eucal}
\usepackage{mathrsfs}
\usepackage{tikz-cd}
\renewcommand{\injlim}{\varinjlim}
\renewcommand{\projlim}{\varprojlim}

\usepackage{enumitem}
\usepackage{booktabs}
\usepackage[pdfusetitle,colorlinks]{hyperref}
\hypersetup{bookmarksdepth=2,pdfencoding=unicode,allcolors=MidnightBlue}

\usepackage{zref-clever}
\zcsetup{abbrev=false,cap=true,nameinlink=false,sort=false,lang=english}
\newcommand{\cref}[1]{\zcref{#1}}
\newcommand{\Cref}[1]{\zcref[S]{#1}}
\zcsetup{pairsep={ and~},lastsep={, and~}}
\zcRefTypeSetup{equation}{Name-sg=,Name-pl=,refbounds={(,,,)}}
\AddToHook{env/equation/begin}{\zcsetup{countertype={equation=equation}}}
\AddToHook{env/align/begin}{\zcsetup{countertype={equation=equation}}}
\zcRefTypeSetup{item}{Name-sg=,Name-pl=,refbounds={(,,,)}}
\newlist{conenum}{enumerate}{1}
\setlist[conenum,1]{label=(\roman*),ref=\roman*}
\zcRefTypeSetup{conenumi}{Name-sg=,Name-pl=,refbounds={(,,,)}}

\NewDocumentCommand{\newzctheorem}{momo}{\IfValueTF{#4}
  {\newtheorem{#1}{#3}[#4]}
  {\IfValueTF{#2}
    {\AddToHook{env/#1/begin}{\zcsetup{countertype={#2=#1}}}\newtheorem{#1}[#2]{#3}}
    {\newtheorem{#1}{#3}}}}
\numberwithin{equation}{section}
\theoremstyle{plain}
\newzctheorem{Theorem}{Theorem}

\zcRefTypeSetup{Theorem}{Name-sg=Theorem,Name-pl=Theorems}
\newzctheorem{theorem}[equation]{Theorem}
\newzctheorem{proposition}[equation]{Proposition}
\newzctheorem{lemma}[equation]{Lemma}
\newzctheorem{corollary}[equation]{Corollary}

\theoremstyle{definition}
\newzctheorem{definition}[equation]{Definition}
\newzctheorem{assumption}[equation]{Assumption}
\zcRefTypeSetup{assumption}{Name-sg=Assumption,Name-pl=Assumptions}
\newzctheorem{example}[equation]{Example}
\newzctheorem{question}[equation]{Question}
\zcRefTypeSetup{question}{Name-sg=Question,Name-pl=Questions}

\theoremstyle{remark}
\newzctheorem{remark}[equation]{Remark}
\newzctheorem{slogan}[equation]{Slogan}
\zcRefTypeSetup{slogan}{Name-sg=Slogan,Name-pl=Slogans}

\let\oldSS\SS\let\SS\relax
\let\oldtop\top\let\top\relax

\newcommand{\FF}{\mathbf{F}}
\newcommand{\NN}{\mathbf{N}}
\newcommand{\ZZ}{\mathbf{Z}}
\newcommand{\QQ}{\mathbf{Q}}
\newcommand{\CC}{\mathbf{C}}
\newcommand{\RR}{\mathbf{R}}
\newcommand{\SS}{\mathbf{S}}

\newcommand{\GG}{\mathbb{G}}

\newcommand{\AH}{\textnormal{AH}}

\newcommand{\Lic}{\textnormal{Lic}}
\newcommand{\Zar}{\textnormal{Zar}}

\newcommand{\con}{\textnormal{con}}
\newcommand{\et}{\textnormal{ét}}
\newcommand{\hol}{\textnormal{hol}}

\newcommand{\lis}{\textnormal{lis}}
\newcommand{\mot}{\textnormal{mot}}
\newcommand{\sm}{\textnormal{sm}}

\newcommand{\top}{\textnormal{top}}
\newcommand{\Bet}{\textnormal{Bet}}

\newcommand{\CH}{\operatorname{CH}}
\newcommand{\Cau}{\operatorname{Cau}}
\newcommand{\D}{\operatorname{D}}
\newcommand{\GL}{\operatorname{GL}}
\newcommand{\Idem}{\operatorname{Idem}}

\newcommand{\ch}{\operatorname{ch}}
\newcommand{\KU}{\operatorname{KU}}

\newcommand{\Pic}{\operatorname{Pic}}

\newcommand{\Spec}{\operatorname{Spec}}
\newcommand{\Sp}{\operatorname{Sp}}
\newcommand{\Tor}{\operatorname{Tor}}
\newcommand{\cdim}{\operatorname{cdim}}
\newcommand{\fil}{\operatorname{fil}}
\newcommand{\gr}{\operatorname{gr}}
\newcommand{\id}{\operatorname{id}}
\newcommand{\ku}{\operatorname{ku}}

\newcommand{\pr}{\operatorname{pr}}
\newcommand{\supp}{\operatorname{supp}}

\newcommand{\X}{\mathord{-}}

\newcommand{\Cat}[1]{\mathsf{#1}}
\newcommand{\shf}[1]{\mathcal{#1}}
\newcommand{\Cls}[1]{\mathscr{#1}}

\title{The Oka principle for étale Chow groups}
\author{Ko Aoki}
\address{Department of Mathematical Sciences,
  University of Copenhagen, Denmark
}
\email{aoki@math.ku.dk}
\date{\today}

\begin{document}

\begin{abstract}
  The celebrated theorems of Shilov, Arens–Royden, and Forster
  give direct descriptions
  of the first three integral cohomology groups
  of the Gelfand spectrum of a commutative complex Banach algebra.
  In his 1974 ICM address,
  Taylor asked whether the higher cohomology groups admit descriptions
  in terms of the underlying ring.
  We give a solution to this question in even degrees:
  The étale (aka Lichtenbaum) Chow group
  in every codimension
  is canonically isomorphic
  to the corresponding even integral cohomology group
  of the Gelfand spectrum.
\end{abstract}

\maketitle

\section{Introduction}\label{s:intro}

In this paper,
rings are by default unital and commutative.
A Banach algebra means a (unital commutative) complex Banach algebra.
For such an algebra~\(A\),
its \emph{Gelfand spectrum}~\(\Sp(A)\) is
the compactum (i.e., compact Hausdorff space) of characters~\(A\to\CC\).
Cohomology of local compacta
is understood in the sheaf (or equivalently,
Čech) cohomology sense.

\subsection{Taylor’s question}\label{ss:taylor}

According to Grauert–Remmert~\cite[page~145]{GrauertRemmert79}, the Oka
principle says that “on a reduced\footnote{The reducedness assumption is not relevant in this paper.
} Stein space, problems which can be
cohomologically formulated have only topological obstructions.”  Its
best-known example is the following:

\begin{theorem}[Grauert]\label{grauert}
  Let~\(U\) be a Stein space. The comparison map
  \begin{equation*}
    H^{1}_{\hol}(U;\GL_{r}(\CC))
    \to
    H^{1}_{\con}(U;\GL_{r}(\CC))
  \end{equation*}
  from holomorphic to continuous nonabelian cohomology is an
  isomorphism for \(r\geq0\).
\end{theorem}

The case~\(r=1\) is due to Oka;
Grauert proved the statement in arbitrary rank.
This immediately implies the following:

\begin{corollary}\label{xsyq8z}
  Let~\(K\) be a Stein compactum.
  Forgetting the holomorphic structure induces an isomorphism
  \begin{equation*}
    K_{0}(\Cls{O}(K))\simeq
    K_{0}(\Cls{C}(K))\mathrel{(\simeq}
    \KU^{0}(K)),
  \end{equation*}
  where \(\Cls{O}(K)\) denotes the ring
  of global sections.
\end{corollary}

The corresponding comparison
for Banach algebras holds in every degree,
when we consider topological \(K\)-theory;
see~\cite{Novodvorskii67,BraddHigson21}:

\begin{theorem}[Novodvorskii]\label{novodvorskii}
  The Gelfand transform~\(A\to\Cls{C}(\Sp(A))\) induces an
  isomorphism
  \begin{equation*}
    K^{\top}_{*}(A)\simeq
    K^{\top}_{*}(\Cls{C}(\Sp(A)))\mathrel{(\simeq}
    \KU^{-{*}}(\Sp(A)))
  \end{equation*}
  for \({*}\geq0\).
\end{theorem}

Taylor’s 1974 ICM address~\cite{Taylor75}
considered questions of this kind.
There are classical direct descriptions
of the first three integral cohomology groups:

\begin{theorem}[Shilov, Arens–Royden, Forster]\label{classical}
  For a Banach algebra~\(A\),
  we have canonical isomorphisms
  \begin{align*}
    \langle \Idem(A)\rangle&\simeq H^{0}_{\Bet}(\Sp(A);\ZZ),&
    A^{\times}/\exp(A)&\simeq H^{1}_{\Bet}(\Sp(A);\ZZ),&
    \Pic(A)&\simeq H^{2}_{\Bet}(\Sp(A);\ZZ),
  \end{align*}
  where \(\langle\Idem(A)\rangle\)
  is the additive subgroup of~\(A\)
  generated by its idempotents,
  and \(\exp(A)\) is the image
  of \({\exp}\colon A\to A^{\times}\).
\end{theorem}

The original references
are~\cite{Shilov53,Arens63,Royden63,Forster74}.
See also~\cite{Taylor76} for Taylor’s account.
The following is~\cite[Problem~1]{Taylor75}:

\begin{question}[Taylor]\label{xqo4z1}
  Is there a description of
  \(H^{n}_{\Bet}(\Sp(A);\ZZ)\) for \(n\geq3\)
  in terms of the structure of~\(A\)?
\end{question}

As Taylor noted on~\cite[page~117]{Taylor75},
this problem has a rational solution:
Novodvorskii’s description of topological \(K\)-theory followed by the
Chern character recovers the direct sums of the even and odd rational
cohomology groups, and Adams operations separate the individual degrees.
This includes both parities, but its odd part uses
the Banach topology through the connected components of~\(\GL_n(A)\).
An integral description remained open.

We formulate a mathematical question
that is more restricted:

\begin{question}\label{taylor}
  Let \(n\geq0\) be an integer.
  Does there exist a
  functor~\(F\) from rings to
  abelian groups preserving filtered colimits,
  together with functorial isomorphisms
  \begin{equation*}
    F(A)\simeq H^n_{\Bet}(\Sp(A);\ZZ)
  \end{equation*}
  when restricted to Banach algebras~\(A\)?
\end{question}

\begin{remark}\label{ar}
  The Shilov and Forster theorems in \cref{classical}
  answer \cref{taylor} for \(n=0\), \(2\), respectively,
  whereas the Arens–Royden theorem
  does not for \(n=1\):
  The subgroup~\(\exp(A)\) uses the Banach topology
  and is not defined from
  the underlying ring alone.
\end{remark}

For a Banach algebra~\(A\), note
the pattern
\begin{align*}
  \CH^{0}(A)&\simeq H^{0}_{\Zar}(\Spec A;\ZZ)
  \simeq\langle\Idem(A)\rangle,&
  \CH^{1}(A)&\simeq\Pic(A).
\end{align*}
However,
this pattern with ordinary Chow groups fails already in weight two,
as we explain in \cref{s:ordinary}.
We instead use the étale version:

\begin{definition}\label{etale}
  For a smooth \(\QQ\)-algebra and \(q\geq0\),
  we consider \(\ZZ(q)^{\et}\),
  the étale sheafification
  of the classical (aka Bloch) motivic cohomology.
  We left Kan extend this functor to obtain
  \(\ZZ(q)^{\Lic}\colon\Cat{Ring}_{\QQ}\to\D(\ZZ)\).
  We define the \emph{étale Chow group}
  as its \(\pi_{-2q}\), i.e.,
  \begin{equation*}
    \CH^q_{\et}(A)=H_{\Lic}^{2q}(\Spec A;\ZZ(q)).
  \end{equation*}
\end{definition}

Our main result is the following,
which answers \cref{taylor}
for every even integer~\(n\):

\begin{Theorem}\label{main}
  Let~\(A\) be a Banach algebra.
  There is a functorial isomorphism
  \begin{equation*}
    \CH_{\et}^{q}(A)
    \simeq
    H^{2q}_{\Bet}(\Sp(A);\ZZ)
  \end{equation*}
  for \(q\geq0\).
\end{Theorem}

Literally as a functor on all rings,
take \(R\mapsto\CH^q_{\et}(R\otimes_{\ZZ}\QQ)\);
it preserves filtered
colimits and restricts
to the functor in \cref{main} on complex Banach algebras.

\subsection{Outline}\label{ss:outline}

We first recall facts about motivic cohomology
in \cref{s:motivic}.
In particular,
we construct a morphism
\begin{equation*}
  \ZZ(q)^{\Lic}(A)
  \to
  \Gamma_{\Bet}(\Sp(A);\ZZ(q)).
\end{equation*}
Now \cref{main} is the claim
that this induces a bijection
on~\(\pi_{-2q}\).
To prove this,
we use the following elementary input
for \(n=-2q\):

\begin{lemma}\label{fiber}
  Let~\(f\colon C\to C'\) be a map of spectra and let~\(n\in\ZZ\).
  Suppose the following:
  \begin{conenum}
    \item\label{i:fib-q}
      The rationalization~\(f\otimes\QQ\) is an isomorphism on
      \(\pi_n\) and surjective on~\(\pi_{n+1}\).
    \item\label{i:fib-p}
      For every prime~\(p\), the reduction~\(f\otimes\SS/p\) is
      injective on~\(\pi_n\), an isomorphism on~\(\pi_{n+1}\), and
      surjective on~\(\pi_{n+2}\).
  \end{conenum}
  Then~\(\pi_n(f)\) is an isomorphism.
\end{lemma}

\begin{proof}
  We consider~\(F=\operatorname{fib}(f)\).
  By \cref{i:fib-p}, the homotopy
  long exact sequence shows
  \(\pi_{n+1}(F/p)=\pi_{n}(F/p)=0\).
  The Bockstein sequence therefore shows that~\(\pi_n(F)\) and
  \(\pi_{n-1}(F)\) are torsion free.
  By \cref{i:fib-q}, we have \(\pi_n(F)\otimes\QQ=0\),
  and hence~\(\pi_n(F)=0\);
  thus \(\pi_n(f)\) is injective.
  Its cokernel is torsion because
  \(\pi_n(f)\otimes\QQ\) is an isomorphism,
  but it also injects into
  the torsion-free group~\(\pi_{n-1}(F)\).
  Thus \(\pi_{n}(f)\) is surjective as well.
\end{proof}

For~\cref{i:fib-q}
of \cref{fiber},
we prove the following
in \cref{s:rational};
note again that \cref{main-q-1}
was already pointed out by Taylor in some form
on~\cite[page~117]{Taylor75}:

\begin{Theorem}\label{main-q-1}
  Let~\(A\) be a Banach algebra.
  The map
  \begin{equation*}
    \CH^{q}_{\et}(A)\otimes\QQ
    \to
    H^{2q}_{\Bet}(\Sp(A);\QQ(q))
  \end{equation*}
  is bijective for \(q\geq0\).
\end{Theorem}

\begin{Theorem}\label{main-q-2}
  Let~\(A\) be a Banach algebra.
  The map
  \begin{equation*}
    H^{2q-1}_{\Lic}(\Spec A;\QQ(q))
    \to
    H^{2q-1}_{\Bet}(\Sp(A);\QQ(q))
  \end{equation*}
  is surjective
  for \(q\geq0\).
\end{Theorem}

For~\cref{i:fib-p}
of \cref{fiber},
we prove the following stronger statement
in \cref{s:finite}:

\begin{Theorem}\label{main-p}
  Let~\(A\) be a Banach algebra
  and \(p\) a prime.
  Then
  \begin{equation*}
    \FF_{p}(q)^{\Lic}(A)\to\Gamma_{\Bet}(\Sp(A);\FF_{p}(q))
  \end{equation*}
  is an equivalence in \(\D(\FF_{p})\)
  for \(q\geq0\).
\end{Theorem}

This concludes the proof of \cref{main}.
In \cref{s:ordinary},
we explain why ordinary Chow groups cannot replace étale Chow groups
in \cref{main}.

\subsection*{Acknowledgments}\label{ss:ack}

I thank Peter Scholze for helpful discussions,
especially for pointing me to~\cite{BouthierCesnavicius22}.
The initial stages of this research were carried out
at the Max Planck Institute for Mathematics.
During the course of this work
I was supported by the Danish National Research Foundation
through the Copenhagen Center for Geometry and Topology (DNRF151).

\subsection*{Convention}\label{ss:con}

We do not use~$\mathrm{L}$ nor~$\mathrm{R}$
to indicate how they are derived.
In particular,
\(\Gamma\) denotes cohomology.

\section{Recollection: motivic cohomology}\label{s:motivic}

We write \(\Cat{Ring}_{\QQ}\)
for the category of (static)\footnote{Most of what follows here works
  in the animated setting,
  but that generality is not relevant
  in this paper.
} \(\QQ\)-algebras.
We write \(\Cat{Ring}_{\QQ}^{\sm}\)
for the full subcategory of smooth \(\QQ\)-algebras.

\begin{remark}\label{xi4u0d}
  Throughout this section,
  we only consider \(\QQ\)-algebras.
  We could start from~\(\ZZ\) as a base
  and then consider the lisse extension
  of the motivic cohomology
  to obtain the same result.
  We stick to this convention for simplicity.
\end{remark}

\subsection{Motivic and Lichtenbaum cohomologies}

The shifted Bloch cycle complex \(z^{q}(\X,{*})[-2q]\)
determines
the \emph{motivic cohomology} functor
\(\ZZ(q)\colon\Cat{Ring}_{\QQ}^{\sm}\to\D(\ZZ)\)
for \(q\geq0\).
We write its étale sheafification
as \(\ZZ(q)^{\et}\colon\Cat{Ring}_{\QQ}^{\sm}\to\D(\ZZ)\),
which is often also called the \emph{Lichtenbaum cohomology}.
Following~\cite{ElmantoMorrow},
we use the term “lisse” in the following definition:

\begin{definition}\label{lisse}
  The \emph{lisse motivic cohomology}~\(\ZZ(q)^{\lis}\)
  is the left Kan extension
  of~\(\ZZ(q)\)
  along the inclusion
  \(\Cat{Ring}_{\QQ}^{\sm}\subset\Cat{Ring}_{\QQ}\).
  Similarly, we write \(\ZZ(q)^{\Lic}\)
  for the left Kan extension of \(\ZZ(q)^{\et}\).
\end{definition}

\begin{remark}\label{xc6217}
  Beware that \(\ZZ(q)^{\lis}\)
  is not even a Zariski sheaf.
  For example,
  as noted in~\cite[Example~3.2]{ElmantoMorrow},
  we have
  \begin{equation}\label{e:weight-one}
    \ZZ(1)^{\lis}(A)
    \simeq
    (\tau_{\geq-1}\Gamma_{\Zar}(\Spec A,\GG_m))[-1].
  \end{equation}
\end{remark}

As a general rule,
we write~\(\pi_{-n}(\ZZ(q)^{\X}(A))\)
as \(H^{n}_{\X}(\Spec A;\ZZ(q))\).
The main subject of this paper
is the Chow part:

\begin{definition}\label{xcasq8}
  We define the \emph{Chow group}
  and the \emph{étale Chow group} as
  \begin{align*}
    \CH^q(A)&=H^{2q}_{\lis}(\Spec A;\ZZ(q)),&
    \CH^q_{\et}(A)&=H^{2q}_{\Lic}(\Spec A;\ZZ(q)).
  \end{align*}
\end{definition}

\begin{example}\label{x4h978}
  By definition,
  \(\CH^{q}\)
  for smooth \(\QQ\)-algebras
  coincides with the usual Chow group.
\end{example}

\begin{example}\label{low-chow}
  In the two lowest weights this gives, for every~\(A\),
  \begin{align*}
    \CH^0(A)&\simeq H^0_{\Zar}(\Spec A;\ZZ),&
    \CH^1(A)&\simeq\Pic(A).
  \end{align*}
  To see this,
  we need to observe that
  these are Kan extended from smooth affines.
  In weight zero, this follows directly:
  A locally constant function
  \(\Spec A\to\ZZ\) has finite image,
  and the functor of those with image in a fixed finite subset
  \(S\subset\ZZ\) is represented by the finite étale
  \(\QQ\)-algebra~\(\QQ^{S}\).
  Thus \(A\mapsto H^0_{\Zar}(\Spec A;\ZZ)\)
  is a filtered colimit
  of functors represented by smooth \(\QQ\)-algebras,
  and is therefore left Kan extended
  from smooth affines.
  The weight-one assertion follows from~\cref{e:weight-one}.
\end{example}

In fact, \(\CH^{q}\) admits a underived definition:

\begin{proposition}\label{top-lan}
  For every~\(\QQ\)-algebra~\(B\) and~\(q\geq0\),
  there is a natural isomorphism
  \begin{equation*}
    \CH^{q}(B)
    \simeq
    \injlim_{A\in (\Cat{Ring}_{\QQ}^{\sm})_{/B}}\CH^{q}(A)
  \end{equation*}
  in \(\Cat{Ab}\).
\end{proposition}

\begin{proof}
  This follows from the fact
  that \(\ZZ(q)(A)\)
  is \((-2q)\)-connective for every
  smooth~\(A\) by definition.
  We used the fact that \((\Cat{Ring}_{\QQ}^{\sm})_{/B}\)
  is sifted (as an \(\infty\)-category),
  since it has finite coproducts, given by tensor product over~\(\QQ\).
\end{proof}

\begin{remark}\label{xxox4h}
  For a quasiprojective \(\QQ\)-scheme~\(X\),
  Fulton’s \emph{Chow cohomology},
  which is written as \(A^{q}(X)\),
  is
  \(\injlim_{X\to Y}\CH^q(Y)\),
  where~\(Y\) runs through smooth quasiprojective \(\QQ\)-schemes;
  see~\cite[Section~3.1]{Fulton75}.
  This is essentially the underived description in \cref{top-lan}.
  One difference is that
  he considers quasiprojective schemes,
  which is subsumed by Jouanolou’s trick.
\end{remark}

On smooth affines,
\(\CH_{\et}^{q}\)
agrees with the usual étale Chow group,
as studied, e.g., in~\cite{RosenschonSrinivas16}.
Unlike \cref{top-lan}, there is in general no formula for
\(\CH^q_{\et}(A)\) as a colimit of étale Chow groups in \(\Cat{Ab}\);
since
\(\Gamma_{\et}(\Spec B;\ZZ(q)^{\et})[2q]\) need not be connective,
lower colimit terms may contribute.

\subsection{The cycle map}\label{ss:betti}

Let \(X\) be a compactum
and \(A\) a smooth \(\QQ\)-algebra
with a map \(A\to\Cls{C}(X)\).
This is equivalent to
a morphism of local compacta
\(X\to (\Spec A)(\CC)\).
By Betti realization,
we have a morphism
\begin{equation*}
  \Gamma_{\et}(\Spec A;\ZZ(q)^{\et})
  \to
  \Gamma_{\Bet}((\Spec A)(\CC);\ZZ(q))
  \to
  \Gamma_{\Bet}(X;\ZZ(q)).
\end{equation*}
These composites are compatible with the morphisms of
\((\Cat{Ring}_{\QQ}^{\sm})_{/\Cls{C}(X)}\).
By taking the colimit,
we obtain a natural map
\begin{equation*}
  \ZZ(q)^{\Lic}(\Cls{C}(X))
  \to
  \Gamma_{\Bet}(X;\ZZ(q)).
\end{equation*}

\begin{definition}\label{x25gqw}
  Let \(A\) be a Banach algebra.
  The \emph{cycle map}\footnote{This name comes from the fact that
    on the Chow part this specializes
    to the classical cycle map.
  } is
  \begin{equation*}
    \ZZ(q)^{\Lic}(A)
    \to
    \ZZ(q)^{\Lic}(\Cls{C}(\Sp(A)))
    \to
    \Gamma_{\Bet}(\Sp(A);\ZZ(q)),
  \end{equation*}
  where the first comes from the Gelfand transform
  and the second is as described above.
\end{definition}

In particular,
we obtain
\begin{equation*}
  \CH^q_{\et}(A)
  \to
  H^{2q}_{\Bet}(\Sp(A);\ZZ(q)).
\end{equation*}
After fixing~\(2\pi i\), we freely identify the topological Tate
twist~\(\ZZ(q)\) with~\(\ZZ\).

\section{Rational coefficients}\label{s:rational}

We prove \cref{main-q-1,main-q-2}
in \cref{ss:rational,ss:odd},
respectively.
We review facts
about rational motivic cohomology in \cref{ss:rat}.

\subsection{Rational motivic cohomology}\label{ss:rat}

First,
rational motivic cohomology is automatically étale:

\begin{proposition}\label{rational}
  For every~\(\QQ\)-algebra~\(A\) and every~\(q\geq0\), there is a
  natural equivalence
  \begin{equation}\label{e:rational}
    \ZZ(q)^{\Lic}(A)\otimes\QQ
    \simeq
    \ZZ(q)^{\lis}(A)\otimes\QQ.
  \end{equation}
\end{proposition}

\begin{proof}
  On smooth \(\QQ\)-algebras,
  \(\QQ(q)=\ZZ(q)\otimes\QQ\) satisfies étale descent.
\end{proof}

Also, the motivic filtration splits
rationally.
Concretely, we note the following:

\begin{proposition}\label{chern}
  For every~\(\QQ\)-algebra~\(A\),
  there is a natural decomposition
  \begin{equation*}
    K_{0}(A)\otimes\QQ
    \simeq
    \bigoplus_{q\geq0}\CH^{q}(A)\otimes\QQ.
  \end{equation*}
  The summand indexed by~\(q\) is the Adams weight-\(q\) summand,
  explained in the proof.
\end{proposition}

\begin{proof}
  If~\(B\) is a smooth~\(\QQ\)-algebra,
  the algebraic Chern character gives a natural isomorphism
  \begin{equation*}
    K_{0}(B)\otimes\QQ
    \simeq
    \bigoplus_{q\geq0}\CH^{q}(B)\otimes\QQ.
  \end{equation*}
  Under this isomorphism,
  \(\psi^{m}\) acts on the summand~\(\CH^{q}(B)\otimes\QQ\)
  as multiplication by~\(m^{q}\).
  The result follows from \cref{top-lan}
  and the counterpart for~\(K_{0}\)
  by Kan extension.
\end{proof}

\subsection{The rational comparison theorem}\label{ss:rational}

We prove \cref{main-q-1},
stating that
\(\CH^{q}_{\et}(A)\otimes\QQ\to H^{2q}_{\Bet}(\Sp(A);\QQ(q))\)
is bijective for \(q\geq0\).
Again note that this was already pointed
out by Taylor;
see~\cite[page~117]{Taylor75}.

\begin{proof}[Proof of \cref{main-q-1}]
  We write~\(X\) for \(\Sp(A)\).
  \Cref{novodvorskii} gives a natural isomorphism
  \begin{equation*}
    K_{0}(A)\simeq\ku^{0}(\Sp(A)).
  \end{equation*}
  Since exterior powers commute with extension of scalars,
  this is compatible with
  the~\(\lambda\)-operations,
  and hence the Adams operations.

  For every compactum~\(X\),
  the topological Chern character gives
  \begin{equation*}
    \ku^{0}(X)\otimes\QQ
    \simeq
    \bigoplus_{r\geq0}H^{2r}_{\Bet}(X;\QQ),
  \end{equation*}
  and~\(\psi^{m}\) acts in degree~\(2r\)
  as multiplication by~\(m^{r}\);
  see, e.g.,~\cite[Theorem~V.3.27]{Karoubi78}.
  Thus its weight-\(q\) eigenspace is~\(H^{2q}_{\Bet}(X;\QQ)\).
  By \cref{chern},
  compatibility of algebraic and topological Chern characters
  (see, e.g.,~\cite[Theorem~28]{FriedlanderWalker05})
  with our comparison morphism shows the desired result.
\end{proof}

\subsection{The other rational input}\label{ss:odd}

We then prove \cref{main-q-2},
stating that
\(H^{2q-1}_{\Lic}(\Spec A;\QQ(q))\to H^{2q-1}_{\Bet}(\Sp(A);\QQ(q))\)
is surjective for \(q\geq0\).

\begin{proof}[Proof of \cref{main-q-2}]
  We write~\(X\) for \(\Sp(A)\).
  We assume~\(q\geq1\),
  since the case \(q=0\) is tautological.

  Let~\(x\in H^{2q-1}_{\Bet}(X;\QQ)\).
  The odd topological Chern character gives
  \begin{equation*}
    \ku^{-1}(X)\otimes\QQ
    \simeq
    \bigoplus_{r\geq1} H^{2r-1}_{\Bet}(X;\QQ).
  \end{equation*}
  Choose~\(\kappa\in\ku^{-1}(X)\otimes\QQ\) whose weight-\(q\)
  component is~\(x\).
  By \cref{novodvorskii},
  the left-hand side is \(K^{\top}_{1}(A)\otimes\QQ\).
  We may therefore write~\(\kappa\) as a finite rational linear
  combination of classes represented by matrices
  \(g_{i}\in\GL_{N_{i}}(A)\).
  Let \(B_i\)
  be the coordinate ring of the affine variety~\(\GL_{N_i,\QQ}\).
  The universal matrix~\(u_{i}\)
  defines a class in~\(K_{1}(B_{i})\),
  and
  its weight-\(q\) motivic Chern-character component is
  \begin{equation*}
    \ch^{\mot}_{q}(u_{i})
    \in H^{2q-1}_{\lis}(\Spec B_{i};\QQ(q)).
  \end{equation*}
  The homomorphism~\(B_{i}\to A\) carrying~\(u_{i}\) to~\(g_{i}\) is an
  object of~\((\Cat{Ring}_{\QQ}^{\sm})_{/A}\), so this class determines
  an element of the source via \cref{e:rational}.
  Compatibility of motivic and topological Chern characters under
  Betti realization sends it to
  \(\ch^{\top}_{q}(g_{i})\);
  see, e.g.,~\cite[Theorem~28]{FriedlanderWalker05}.
  The same rational linear combination
  of these universal classes maps
  to~\(x\),
  proving surjectivity.
\end{proof}

\begin{remark}\label{odd-kernel}
  The map in \cref{main-q-2} need not be injective,
  already in weight one.
  Indeed, for the Banach algebra~\(A=\CC\),
  \cref{e:weight-one} and
  \cref{e:rational} give
  \begin{equation*}
    H^1(\ZZ(1)^{\Lic}(\CC)\otimes\QQ)
    \simeq \CC^\times\otimes\QQ,
  \end{equation*}
  whereas the target is
  \(H^1_{\Bet}(\star;\QQ(1))=0\).
\end{remark}

\section{Finite coefficients}\label{s:finite}

We prove \cref{main-p} in \cref{ss:finite}.
We first review facts about finite-coefficient étale motivic cohomology
in \cref{ss:et}.
We give axiomatic arguments in \cref{ss:rig}.

\subsection{Finite-coefficient étale motivic cohomology}\label{ss:et}

\begin{proposition}\label{mod-p}
  For every prime~\(p\), every~\(q\geq0\),
  and every \(\QQ\)-algebra~\(A\),
  there is a natural equivalence
  \begin{equation*}
    \ZZ(q)^{\Lic}(A)\otimes_{\ZZ}\FF_{p}
    \simeq
    \Gamma_{\et}(\Spec A;\mu_p^{\otimes q}).
  \end{equation*}
\end{proposition}

We need the following in the proof:

\begin{lemma}[Bhatt–Mathew]\label{etale-lan}
  Let~\(p\) be a prime and~\(q\geq0\).
  \begin{align*}
    \Cat{Ring}_{\QQ}&\to\D(\ZZ);&
    A&\mapsto
    \Gamma_{\et}(\Spec A;\mu_{p}^{\otimes q})
  \end{align*}
  is left Kan extended from \(\Cat{Ring}_{\QQ}^{\sm}\).
\end{lemma}

\begin{proof}
  In the proof of~\cite[Theorem~5.1]{BhattMathew23},
  Bhatt–Mathew showed that
  the functor
  \begin{align*}
    \Cat{Ring}_{\ZZ}&\to\D(\ZZ);&
    A&\mapsto
    \Gamma_{\et}(\Spec A[1/p];\mu_{p}^{\otimes q})
  \end{align*}
  is left Kan extended from smooth~\(\ZZ\)-algebras.
  The result follows from this statement.
\end{proof}

\begin{proof}[Proof of \cref{mod-p}]
  For smooth~\(B\), the finite-coefficient calculation for the étale
  motivic complex gives
  \begin{equation*}
    \Gamma_{\et}(\Spec B;\ZZ(q)^{\et})\otimes_{\ZZ}\FF_{p}
    \simeq
    \Gamma_{\et}(\Spec B;\mu_p^{\otimes q});
  \end{equation*}
  see, e.g.,~\cite[Theorem~10.3]{MazzaVoevodskyWeibel06}.
  Hence,
  the result follows from \cref{etale-lan}.
\end{proof}

\subsection{Continuity for henselian-invariant functors}\label{ss:rig}

The continuity arguments we present here
are not specific to étale cohomology.
We formulate them axiomatically for future reference.

\begin{definition}\label{henselian-invariant}
  Let~\(R\) be a ring
  and consider a functor \(F\colon\Cat{Ring}_{R}\to\Cat{Set}\).
  We call it \emph{finitary} if it commutes with filtered colimits.
  We call it \emph{henselian invariant}
  if, for every henselian pair~\((A,I)\) of~\(R\)-algebras,
  \(F(A)\to F(A/I)\) is bijective.
\end{definition}

Here, recall that a ring with an ideal \((A,I)\)
is called \emph{henselian}
if \(1+I\subset A^{\times}\) and
when
\begin{equation}
  \label{e:hensel}
  P_{a}(T)=T^{d}(T-1)+a_{d}T^{d}+\dotsb+a_{0}
  \in A[T]
\end{equation}
satisfies \(a_{0}\), \dots, \(a_{d}\in I\),
there is a root~\(t\in 1+I\).

In this section,
our generality is Banach rings over~\(\ZZ\),
as studied notably in~\cite[Section~1.1]{Berkovich90}\footnote{More precisely, he assumes \(\lVert1\rVert=1\).
  Here, we only impose~\(\lVert1\rVert\leq1\)
  to allow the zero ring.
}.
A \emph{seminorm} on a ring~\(A\)
is a map \(A\to[0,\infty)\)
satisfying
\(\lVert0\rVert\leq0\),
\(\lVert-a\rVert=\lVert a\rVert\),
\(\lVert a+b\rVert\leq\lVert a\rVert+\lVert b\rVert\),
\(\lVert1\rVert\leq1\),
and \(\lVert ab\rVert\leq\lVert a\rVert\lVert b\rVert\).
It is a \emph{norm} if \(\lVert a\rVert=0\) implies \(a=0\).
A \emph{Banach ring} is a ring equipped with a complete norm.
Morphisms of Banach rings are contractive ring homomorphisms.

\begin{definition}\label{cauchy-rings}
  For \(r\geq0\),
  we equip~\(\NN^{r}\) with the product order.
  For~\(n\in\NN^{r}\),
  we write \(\NN^{r}_{\geq n}\) for the corresponding tail.
  We call a map
  \(a\colon\NN^{r}_{\geq n}\to A\) \emph{cubically Cauchy}
  if its restriction to every coordinate face obtained
  by fixing any subset of the coordinates is Cauchy.
  This includes the full \(r\)-dimensional face itself.
  E.g., for \(r=2\), the array is Cauchy for the product order
  and every row and every column is Cauchy.
  Let \(\Cau^{r}_{n}(A)\) be the ring of cubically Cauchy nets,
  with pointwise operations, and set
  \begin{equation*}
    \Cau^{r}_{\infty}(A)=\injlim_{n\in\NN^{r}}\Cau^{r}_{n}(A),
  \end{equation*}
  where the maps are restrictions,
  to be the ring of germs of cubically Cauchy nets.
  We write \(\Cau^{r}_{\infty}(A)_{0}\)
  for the ideal of germs converging to zero.
\end{definition}

Pointwise ring operations preserve the condition, as may be checked on
each coordinate face.
For \(r\geq1\), taking the joint limit in the separated
completion gives an isomorphism
\begin{equation*}
  \Cau^{r}_{\infty}(A)/\Cau^{r}_{\infty}(A)_{0}\simeq A^{\wedge}.
\end{equation*}
Indeed, the diagonal is cofinal,
while every ordinary Cauchy sequence gives a cubically Cauchy array
by pullback along a coordinate projection.

We use the following,
which is extracted
from the argument of Gabber in
the proof of~\cite[Theorem~2.1.15]{BouthierCesnavicius22}\footnote{However,
  see \cref{xxf2s8} for an error in the proof
  presented there.
}:

\begin{proposition}[Gabber]\label{gabber}
  Let~\(R\) be a ring
  and \(A\) an \(R\)-algebra
  equipped with a seminorm.
  Suppose that \((\Cau^{r}_{\infty}(A),\Cau^{r}_{\infty}(A)_{0})\)
  is a henselian pair for~\(r=1\), \(2\).
  Then (separated) completion induces a bijection~\(F(A)\simeq F(A^{\wedge})\)
  for any finitary henselian-invariant functor~\(F\colon\Cat{Ring}_{R}\to\Cat{Set}\).
\end{proposition}

\begin{proof}
  For \(n\geq0\),
  we consider
  \(t\colon\NN\to\star\),
  and projections \(\pr_{1}\), \(\pr_{2}\colon\NN_{\geq(n,n)}^{2}
  \simeq(\NN_{\geq n})^{2}\rightrightarrows\NN_{\geq n}\)
  and form the diagram
  \begin{equation*}
    \begin{tikzcd}
      A\ar[d,"t^{\star}"']\ar[r,"t^{\star}"]&
      \Cau^{1}_{n}(A)\ar[d,"\pr_{1}^{\star}"]\\
      \Cau^{1}_{n}(A)\ar[r,"\pr_{2}^{\star}"]&
      \Cau^{2}_{(n,n)}(A)\rlap.
    \end{tikzcd}
  \end{equation*}
  Let \(e_n\colon\Cau^1_n(A)\to A\) be evaluation at~\(n\).
  Restricting a cubically Cauchy array to
  \(\NN_{\geq n}\times\{n\}\) and to
  \(\{n\}\times\NN_{\geq n}\) gives maps
  \(r_1\), \(r_2\colon\Cau^2_{(n,n)}(A)\to\Cau^1_n(A)\).
  They satisfy
  \begin{align*}
    r_i\pr_i^\star&=\id,&
    r_1\pr_2^\star&=t^\star e_n,&
    r_2\pr_1^\star&=t^\star e_n.
  \end{align*}
  Together with \(e_n t^\star=\id\), these identities exhibit the
  displayed square as a split, hence absolute, pullback.
  After applying~\(F\) and taking the colimit over~\(n\),
  using that filtered colimits commute with finite limits in~\(\Cat{Set}\),
  we obtain a pullback square
  \begin{equation*}
    \begin{tikzcd}
      F(A)\ar[d]\ar[r]&
      F(\Cau^{1}_{\infty}(A))\ar[d]\\
      F(\Cau^{1}_{\infty}(A))\ar[r]&
      F(\Cau^{2}_{\infty}(A))\rlap,
    \end{tikzcd}
  \end{equation*}
  where all the maps are injective.
  By the completion identification above and henselian invariance,
  the other three terms identify with~\(F(A^{\wedge})\),
  and the cospan becomes its tautological identity diagram.
  Therefore \(F(A)\to F(A^{\wedge})\) is a bijection.
\end{proof}

\begin{remark}\label{xxf2s8}
  The facewise condition in \cref{cauchy-rings} is needed
  for the splitting argument in the preceding proof.
  In the proof
  of~\cite[Theorem~2.1.15\,(b)]{BouthierCesnavicius22},
  a Cauchy net on~\(S\times S\) is restricted
  to \(\{U\}\times S_{\geq U}\),
  but this restriction need not be Cauchy in general.
\end{remark}

We give two applications of this.
We start with the following useful estimate:

\begin{lemma}\label{hensel}
  For every~\(d\geq0\), there are constants
  \(\epsilon_{d}\), \(r_{d}\), \(c_{d}>0\) with the following property:
  Let~\(A\) be a Banach ring
  and \(a=(a_{0},\dotsc,a_{d})\) be a tuple in~\(A\).
  If \(\lVert a\rVert=\max_i\lVert a_i\rVert<\epsilon_d\), then
  the polynomial~\(P_a\) in~\cref{e:hensel} has a unique root~\(t(a)\) satisfying
  \(\lVert t(a)-1\rVert\leq r_d\), and
  \begin{equation}\label{e:hensel-bound}
    \lVert t(a)-1\rVert\leq c_d\lVert a\rVert.
  \end{equation}
  Moreover, if~\(f\colon A\to A'\) is a morphism of Banach rings
  and~\(a'\) is another such tuple in~\(A'\),
  then
  \begin{equation}\label{e:hensel-diff}
    \lVert f(t(a))-t(a')\rVert
    \leq c_d\max_i\lVert f(a_i)-a'_i\rVert.
  \end{equation}
\end{lemma}

\begin{proof}
  We write~\(T=1+u\) and consider the self-map of~\(A\)
  given by
  \begin{equation*}
    \Phi_a(u)=u-P_a(1+u).
  \end{equation*}
  For~\(a=0\), the polynomial
  \(\Phi_0(u)=u-(1+u)^{d}u\) has neither a constant nor a linear term.
  The identities
  \begin{equation*}
    u^j-v^j=(u-v)\sum_{k=0}^{j-1}u^{k}v^{j-1-k}
  \end{equation*}
  give, for coefficient tuples~\(a\), \(b\), \(0<r\leq1\), and
  \(\lVert u\rVert\), \(\lVert v\rVert\leq r\),
  \begin{equation}\label{e:contraction-bound}
    \lVert\Phi_a(u)-\Phi_b(v)\rVert
    \leq
    (\alpha_d(r)+\beta_d\lVert a\rVert)\lVert u-v\rVert
    +\beta_d\lVert a-b\rVert,
  \end{equation}
  where \(\lVert a-b\rVert=\max_i\lVert a_i-b_i\rVert\),
  \begin{align*}
    \alpha_d(r)&=\sum_{j=1}^{d}\binom{d}{j}(j+1)r^j,&
    \beta_d&=\sum_{i=0}^{d}(i+1)2^i.
  \end{align*}
  Here, we used
  \(\Phi_0(u)=-\sum_{j=1}^{d}\binom{d}{j}u^{j+1}\).
  We also have
  \begin{equation*}
    \lVert\Phi_a(0)\rVert\leq(d+1)\lVert a\rVert.
  \end{equation*}
  Choose~\(r_d>0\) such that \(r_d\leq1\) and
  \(\alpha_d(r_d)\leq1/4\),
  and then choose~\(\epsilon_d>0\) such that
  \begin{align*}
    \beta_d\epsilon_d&\leq1/4,&
    (d+1)\epsilon_d&\leq r_d/2.
  \end{align*}
  It follows that every~\(\Phi_a\) carries the closed ball of
  radius~\(r_d\) to itself and is \(1/2\)-Lipschitz there.

  The contraction theorem
  for the complete metric space underlying~\(A\)
  gives a unique fixed point~\(u_a\) in this ball.
  Then \(t(a)=1+u_a\) is precisely the required root of~\(P_a\).
  Moreover,
  \begin{equation*}
    \lVert u_a\rVert
    \leq\tfrac12\lVert u_a\rVert+(d+1)\lVert a\rVert,
  \end{equation*}
  which gives~\cref{e:hensel-bound} with \(c_d=2(d+1)\).

  Finally, we consider \(u=f(t(a))-1\) and \(u'=t(a')-1\).
  Since~\(f\) is contractive, these belong to the same closed ball
  and are fixed points of~\(\Phi_{f(a)}\) and~\(\Phi_{a'}\), respectively.
  Thus~\cref{e:contraction-bound} gives
  \begin{equation*}
    \lVert u-u'\rVert
    \leq\tfrac12\lVert u-u'\rVert
    +\beta_d\max_i\lVert f(a_i)-a'_i\rVert.
  \end{equation*}
  Increasing~\(c_d\) to \(2\max(d+1,\beta_d)\)
  gives~\cref{e:hensel-diff}.
\end{proof}

\begin{theorem}\label{banach}
  Let~\(R\) be a ring and let
  \(F\colon\Cat{Ring}_R\to\Cat{Set}\)
  be a finitary and henselian-invariant functor.
  Consider a filtered diagram of Banach rings whose
  underlying diagram takes values in~\(\Cat{Ring}_R\).
  If \(A=\injlim_i A_i\) is its colimit \emph{in Banach rings},
  then
  \begin{equation*}
    \injlim_{i}F(A_{i})\to F(A)
  \end{equation*}
  is a bijection.
\end{theorem}

\begin{proof}
  We apply \cref{gabber}
  for the \emph{algebraic} colimit~\(B\)
  equipped with the colimit seminorm.
  Since \(A\) is the completion of~\(B\),
  it suffices to show that
  \((\Cau^{r}_{\infty}(B),\Cau^{r}_{\infty}(B)_{0})\)
  is henselian for any~\(r\geq0\).

  We first show that
  \(\Cau^{r}_{\infty}(B)_{0}\) is contained in the Jacobson radical.
  Let~\((a_n)_n\) be a cubically Cauchy net in~\(B\) converging to zero.
  Fix \(0<\rho<1\). On a sufficiently late tail,
  each~\(a_n\) can be lifted to some Banach stage with norm at most~\(\rho\).
  The geometric series in that stage gives an inverse~\(b_n\in B\)
  to~\(1+a_n\), with \(\lVert b_n\rVert\leq(1-\rho)^{-1}\).
  The identity
  \begin{equation*}
    b_n-b_m=b_n(a_m-a_n)b_m
  \end{equation*}
  and the uniform bound show that~\((b_n)_n\) is Cauchy.
  Applying the same estimate after restriction to any coordinate face
  shows that it is Cauchy on every coordinate face.
  Its germ is the inverse of~\(1+(a_n)_n\).

  It remains to verify the root condition.
  Consider~\cref{e:hensel} with
  \(a_{i}\in\Cau^{r}_{\infty}(B)_{0}\),
  and choose cubically Cauchy representatives~\(a_{i,n}\in B\)
  converging to zero.
  On a sufficiently late tail,
  lift the tuple~\((a_{0,n},\dotsc,a_{d,n})\)
  to a common Banach stage with norm less than~\(\epsilon_d\),
  and let~\(t_n\in B\) be the image of the root supplied by
  \cref{hensel}.
  Uniqueness in a common later stage shows that~\(t_n\)
  is independent of the choices.
  Taking lifts whose norms are arbitrarily close to the colimit
  seminorm gives
  \begin{align*}
    \lVert t_n-1\rVert
    &\leq c_d\max_i\lVert a_{i,n}\rVert,&
    \lVert t_n-t_m\rVert
    &\leq c_d\max_i\lVert a_{i,n}-a_{i,m}\rVert.
  \end{align*}
  These estimates show that~\((t_n)_n\) is Cauchy and converges
  to~\(1\). Applying the second estimate after restriction to
  any coordinate face shows that it is Cauchy on every coordinate face.
  Its germ is the desired root in
  \(1+\Cau^{r}_{\infty}(B)_{0}\).
\end{proof}

\begin{theorem}\label{calculus}
  Let~\(R\) be a ring and
  \(F\colon\Cat{Ring}_{R}\to\Cat{Set}\) a finitary henselian-invariant functor.
  Let~\(A\subset B\) be a dense \(R\)-subalgebra of a Banach ring.
  Suppose the following:
  \begin{conenum}
    \item
      If~\(z\in A\) and \(\lVert z\rVert<1\), then the inverse
      \((1+z)^{-1}\), computed in~\(B\), belongs to~\(A\).
    \item
      For every~\(d\geq0\) and every
      \(a_{0}\), \dots, \(a_{d}\in A\) with
      \(\max_{i}\lVert a_{i}\rVert<\epsilon_{d}\), the distinguished root
      \(t(a)\in B\) of~\cref{e:hensel} supplied by \cref{hensel} belongs to~\(A\).
  \end{conenum}
  Then the inclusion induces a bijection~\(F(A)\simeq F(B)\).
\end{theorem}

\begin{proof}
  We equip~\(A\) with the norm induced from~\(B\)
  and apply \cref{gabber}.
  It suffices to show that
  \((\Cau^{r}_{\infty}(A),\Cau^{r}_{\infty}(A)_{0})\) is henselian
  for any~\(r\geq0\).

  We now argue as in the proof of \cref{banach}.
  The first assumption ensures that the termwise inverses constructed there
  belong to~\(A\); the geometric-series bound in~\(B\) and the same
  difference identity show that they are Cauchy, and the same
  estimate after restriction to any coordinate face shows that they are
  Cauchy on every coordinate face.
  Similarly, the second assumption ensures that the distinguished roots
  constructed there belong to~\(A\), and
  \cref{e:hensel-bound,e:hensel-diff} show that they are Cauchy
  and converge to~\(1\).
  Applying the same estimates after restriction
  to any coordinate face proves Cauchyness on every coordinate face.
  Their germs verify the Jacobson radical
  and root conditions, respectively.
\end{proof}

For complex Banach algebras,
this gives the following:

\begin{corollary}\label{dense}
  Let~\(F\colon\Cat{Ring}_{\CC}\to\Cat{Set}\) be finitary and henselian
  invariant. Let~\(A\to B\) be a morphism of complex
  Banach algebras with dense image. If it induces a homeomorphism
  \(\Sp(B)\simeq\Sp(A)\),
  then~\(F(A)\to F(B)\) is a bijection.
\end{corollary}

\begin{proof}
  We consider~\(I=\ker(A\to B)\).
  Every character of~\(A\) factors through~\(B\),
  so~\(I\) is contained in the Jacobson radical of~\(A\).
  The holomorphic functional calculus in~\(A\),
  applied to the local root
  of~\cref{e:hensel},
  shows that~\((A,I)\) is henselian.
  The induced injection \(A/I\to B\) is spectral.
  Since~\(A/I\) is a Banach algebra,
  naturality of joint holomorphic functional calculus~\cite{Taylor72}
  shows that its image in~\(B\) satisfies
  the two closure conditions of \cref{calculus}.
  Now we see \(F(A)\simeq F(A/I)\simeq F(B)\).
\end{proof}

\begin{corollary}\label{holomorphic}
  Let~\(K\subset\CC^n\) be a polynomially convex compact subset.
  We write \(\Cls{O}(K)\) for the ring of germs of
  holomorphic functions along~\(K\)
  and \(\Cls{A}(K)\) for its completion\footnote{In this case,
    Oka–Weil identifies this completion with the uniform closure on~\(K\)
    of the polynomial functions, which is typically written as~\(\Cls{B}(K)\).
  };
  see, e.g.,~\cite[Section~6]{k-ros-3}.
  Let~\(F\colon\Cat{Ring}_{\CC}\to\Cat{Set}\) be
  a finitary henselian-invariant functor.
  Then completion induces a bijection
  \(F(\Cls{O}(K))\simeq F(\Cls{A}(K))\).
\end{corollary}

\begin{proof}
  The map~\(\Cls{O}(K)\to\Cls{A}(K)\) has dense image by definition.
  We consider
  \(I=\ker(\Cls{O}(K)\to\Cls{A}(K))\).
  If~\(z\in\Cls{O}(K)\) satisfies
  \(\lVert z\rVert_{K}<1\),
  after shrinking the neighborhood of~\(K\)
  on which~\(z\) is defined,
  \(1+z\) has no zeros there.
  Hence \((1+z)^{-1}\) is again a holomorphic germ along~\(K\).

  Similarly, near the origin in the coefficient space,
  the holomorphic implicit function theorem
  gives the distinguished root
  \(h(a_{0},\dotsc,a_{d})\) of~\cref{e:hensel}.
  For a tuple of germs
  whose supremum norm on~\(K\) is less than~\(\epsilon_{d}\),
  shrinking their common domain makes the composite~\(h(a_{0},\dotsc,a_{d})\)
  a holomorphic germ along~\(K\).
  Thus the image of~\(\Cls{O}(K)/I\) in~\(\Cls{A}(K)\)
  has the two closure properties in \cref{calculus}.
  The same argument for coefficient tuples in~\(I\),
  which vanish on~\(K\),
  shows that~\((\Cls{O}(K),I)\) is henselian.
  Therefore, \cref{calculus} gives
  \(F(\Cls{O}(K))\simeq F(\Cls{O}(K)/I)\simeq F(\Cls{A}(K))\).
\end{proof}

\subsection{Application to finite étale cohomology}\label{ss:finite}

We now specialize the preceding results.
Let~\(\Lambda\) be a finite abelian group.
For every~\(r\geq0\), the functor
\begin{align*}
  \Cat{Ring}&\to\Cat{Set};&
  A&\mapsto H^r_{\et}(\Spec A;\Lambda)
\end{align*}
is finitary.
It is furthermore
henselian invariant by the affine analog of proper base change,
proved by Gabber~\cite{Gabber94} and Huber~\cite{Huber93}.

Consider a Banach algebra~\(A\).
For every finitely generated~\(\CC\)-subalgebra \(B\subset A\),
restriction of characters gives a continuous map
\(
  \Sp(A)\to(\Spec B)(\CC)
\).
Composing the usual étale--Betti comparison with pullback along this map,
and taking the filtered colimit over~\(B\), gives
\(
  \Gamma_{\et}(\Spec A;\Lambda)
  \to\Gamma_{\Bet}(\Sp(A);\Lambda)
\).
Combining the discussion in the previous section
with Benoist’s calculation in~\cite{Benoist25}
gives the following:

\begin{theorem}\label{xhppcz}
  Let \(A\) be a Banach algebra
  and \(\Lambda\) a finite abelian group.
  Then the comparison morphism
  constructed above
  \begin{equation*}
    \Gamma_{\et}(\Spec A;\Lambda)
    \to\Gamma_{\Bet}(\Sp(A);\Lambda)
  \end{equation*}
  is an equivalence.
\end{theorem}

\begin{proof}
  We write~\(X\) for \(\Sp(A)\).
  Suppose first that~\(A\)
  is topologically generated by~\(a_{1}\), \dots,~\(a_{n}\).
  \begin{align*}
    X&\to\CC^{n};&
    \chi&\mapsto(\chi(a_{1}),\dotsc,\chi(a_{n}))
  \end{align*}
  is a topological embedding because the~\(a_i\) topologically generate~\(A\).
  Its image is polynomially convex:
  If~\(z\) belongs to its polynomial hull,
  then, for every polynomial~\(P\),
  \begin{equation*}
    \lvert P(z)\rvert
    \leq\sup_{\chi\in X}
    \lvert P(\chi(a_{1}),\dotsc,\chi(a_{n}))\rvert
    \leq\lVert P(a_{1},\dotsc,a_{n})\rVert_{A},
  \end{equation*}
  so evaluation at~\(z\) extends to a character of~\(A\).
  Thus~\(z\) belongs to the image.
  We identify~\(X\) with this image.
  Under this identification, the Gelfand transform factors as
  \begin{equation*}
    A\to\Cls{A}(X)\to\Cls{C}(X).
  \end{equation*}
  The first map has dense image and induces a homeomorphism
  on maximal ideal spaces.
  Applying \cref{dense,holomorphic} to finite étale cohomology,
  we see
  \begin{equation*}
    \Gamma_{\et}(\Spec A;\Lambda)
    \simeq
    \Gamma_{\et}(\Spec\Cls{A}(X);\Lambda)
    \simeq
    \Gamma_{\et}(\Spec\Cls{O}(X);\Lambda).
  \end{equation*}
  The compactum~\(X\) is a Stein compactum.
  Benoist’s comparison theorem therefore
  identifies the last term with
  \(\Gamma_{\Bet}(X;\Lambda)\);
  see~\cite[Theorem~6.1]{Benoist25}.
  By naturality, the resulting equivalence is the comparison map
  in the statement.

  For general~\(A\), let~\(A_{i}\) run through its closed
  topologically finitely generated subalgebras.
  Their union is~\(A\) as an ordinary ring.
  Thus we have
  \(
    \injlim_{i}H^{r}_{\et}(\Spec A_{i};\Lambda)
    \simeq H^{r}_{\et}(\Spec A;\Lambda)
  \).
  On the topological side, standard continuity of sheaf cohomology for
  \(\Sp(A)\simeq\projlim_i\Sp(A_i)\) gives the corresponding statement.
  The result follows from the topologically finitely generated case.
\end{proof}

\begin{proof}[Proof of \cref{main-p}]
  After choosing a primitive \(p\)th root of unity,
  \(\mu_p^{\otimes q}\) is the constant étale sheaf~\(\FF_p\)
  on complex schemes.
  By \cref{mod-p}, the source in \cref{main-p} is
  \(\Gamma_{\et}(\Spec A;\mu_p^{\otimes q})\).
  The comparison in \cref{xhppcz} identifies it with
  \(\Gamma_{\Bet}(\Sp(A);\FF_p(q))\).
\end{proof}

\section{A remark on ordinary Chow groups}\label{s:ordinary}

Here, we explain the promised counterexample,
which shows that
the ordinary Chow groups do not work,
even for the C*-algebra case.

\begin{definition}\label{xx7m8a}
  Let~\(B\) be a smooth \(\QQ\)-algebra. Under the identification of
  \(K_{0}(B)\) with the Grothendieck group of coherent \(B\)-modules
  (see, e.g.,~\cite[Tags~0FDH and~0FDI]{SP}),
  its \emph{codimension filtration}
  \(
    \fil^{q}_{\cdim}K_{0}(B)
  \) is defined as the subgroup
  generated by the classes~\([M]\)
  of coherent modules~\(M\)
  satisfying \(\cdim(\supp M)\geq q\).
\end{definition}

\begin{definition}\label{xx0p2x}
  Let~\(X\) be a compactum.
  Its \emph{Atiyah–Hirzebruch filtration} is
  defined as
  \begin{equation*}
    \fil^{q}_{\AH}\ku^{0}(X)
    =\operatorname{im}((\tau_{\geq2q}\ku)^{0}(X)
    \to\ku^{0}(X)).
  \end{equation*}
\end{definition}

We observe a certain compatibility between these two filtrations:

\begin{proposition}\label{fils}
  Let \(A\) be a smooth \(\QQ\)-algebra
  and \(X\to(\Spec A)(\CC)\)
  be a continuous map from a compactum.
  Then \(K_{0}(A)\to\ku^{0}(X)\) determines
  a (nonstrict) morphism of filtered abelian groups
  \begin{equation*}
    \fil^{*}_{\cdim}K_{0}(A)\to
    \fil^{*}_{\AH}\ku^{0}(X).
  \end{equation*}
\end{proposition}

We need the following equivalent description
of the Atiyah–Hirzebruch filtration:

\begin{lemma}\label{xdm1cd}
  For a finite CW~complex~\(K\),
  \begin{equation*}
    \fil^{q}_{\AH}\ku^{0}(K)=\ker\bigl(
    \ku^{0}(K)
    \to
    \ku^{0}\bigl(K^{(2q-1)}\bigr)
    \bigr),
  \end{equation*}
  where \(K^{(d)}\) denotes the \(d\)-skeleton of~\(K\).
\end{lemma}

\begin{proof}
  More generally, the fiber sequence
  \(\tau_{\geq d+1}E\to E\to\tau_{\leq d}E\)
  shows that the image of \((\tau_{\geq d+1}E)^{0}(K)\)
  in~\(E^{0}(K)\) is the kernel of restriction to~\(K^{(d)}\).
\end{proof}

\begin{proof}[Proof of \cref{fils}]
  Let \(M\) be a coherent \(A\)-module whose support~\(Z\)
  has codimension at least~\(q\).
  We show that \([M]\) maps to \(\fil^{q}_{\AH}\ku^{0}(X)\).
  Consider its finite projective resolution.
  After Betti realization,
  this resolution is an exact complex of vector
  bundles on~\((\Spec A)(\CC)\setminus Z(\CC)\).
  Consequently,
  the image restricts to zero there.

  The smooth affine variety~\((\Spec A)(\CC)\)
  has the homotopy type of a finite CW~complex~\(K\).
  Choose a homotopy equivalence \(f\colon K\to(\Spec A)(\CC)\).
  Every stratum
  in a Whitney stratification of~\(Z(\CC)\)
  has real codimension at least~\(2q\).
  Stratified general position therefore homotopes
  \(f\rvert_{K^{(2q-1)}}\)
  to a map with image in \((\Spec A)(\CC)\setminus Z(\CC)\).
  Hence \cref{xdm1cd} gives the desired result on~\(K\).
  Naturality and pullback along \(X\to(\Spec A)(\CC)\) transfer the conclusion from
  the finite CW~model to~\(X\).
\end{proof}

The link with Bloch’s Chow groups
is the following standard consequence
of Riemann–Roch without denominators;
see~\cite[Example~15.3.5]{Fulton98}.

\begin{proposition}\label{x4oj1b}
  Let \(A\) be a smooth \(\QQ\)-algebra.
  For \(q\geq1\), the \(q\)th Chern class vanishes on
  \(\fil^{q+1}_{\cdim}K_{0}(A)\) and hence induces
  \begin{equation*}
    c_{q}\colon
    \gr^{q}_{\cdim}K_{0}(A)\to\CH^{q}(A).
  \end{equation*}
  The assignment \([Z]\mapsto[\shf O_{Z}]\) gives a surjection in the
  other direction, and their composite on~\(\CH^{q}(A)\) is
  multiplication by~\((-1)^{q-1}(q-1)!\).
  In particular, \(c_{2}\) is an isomorphism.
\end{proposition}

\begin{corollary}\label{xxjrs2}
  For a Banach algebra~\(A\),
  the map
  \(\CH^{2}(A)\to H^{4}_{\Bet}(\Sp(A);\ZZ(2))\)
  factors through
  \begin{equation*}
    \gr^{2}_{\AH}\ku^{0}(\Sp(A)).
  \end{equation*}
\end{corollary}

\begin{proof}
  We write~\(X\) for \(\Sp(A)\).
  For every smooth \(\QQ\)-algebra~\(B\) equipped
  with a homomorphism \(B\to A\),
  \cref{fils,x4oj1b} give a map
  \begin{equation*}
    \CH^{2}(B)
    \xrightarrow{c_{2}^{-1}}
    \gr^{2}_{\cdim}K_{0}(B)
    \to
    \gr^{2}_{\AH}\ku^{0}(X).
  \end{equation*}
  We then take the colimit and use
  \cref{top-lan} to obtain
  \(\CH^{2}(A)\to\gr^{2}_{\AH}\ku^{0}(X)\).
  Its composite with the topological Chern class
  \(c_{2}\colon\gr^{2}_{\AH}\ku^{0}(X)\to H^{4}_{\Bet}(X;\ZZ(2))\)
  is the cycle map.
\end{proof}

We finally come to the following:

\begin{example}\label{counterexample}
  We construct a compactum~\(X\)
  such that \(\gr^{2}_{\AH}\ku^{0}(X)\) vanishes
  but \(H^{4}_{\Bet}(X;\ZZ)\) does not.
  By \cref{xxjrs2}, this shows that the cycle map
  \(\CH^{2}(\Cls{C}(X))\to H^{4}_{\Bet}(X;\ZZ)\)
  is not surjective.

  We consider
  \begin{align*}
    W&=\RR P^{2}\wedge\RR P^{4},&
    X&=\Sigma W.
  \end{align*}
  The integral Künneth theorem and the universal coefficient
  theorem give
  \begin{equation*}
    \widetilde H_{\Bet}^{*}(X;\ZZ)\simeq
    \begin{cases}
      \ZZ/2&{*}=4,5,6,7,\\
      0&\text{otherwise.}
    \end{cases}
  \end{equation*}
  On the other hand, the standard calculation
  \(\widetilde\ku\vphantom\ku^{0}(\RR P^{2m})\simeq\ZZ/2^{m}\)
  and
  \(\widetilde\ku\vphantom\ku^{1}(\RR P^{2m})=0\)
  and the Künneth theorem for~\(\KU\) show that
  \begin{equation*}
    \widetilde\ku\vphantom\ku^{0}(X)
    \simeq\widetilde\ku\vphantom\ku^{1}(W)
    \simeq\Tor_{1}^{\ZZ}(\ZZ/2,\ZZ/4)
    \simeq\ZZ/2.
  \end{equation*}

  We use Adams indexing for the Atiyah–Hirzebruch spectral sequence:
  \begin{equation*}
    E_{1}^{s,t}=\widetilde H_{\Bet}^{s+t}(X;\ZZ(s))
    \Rightarrow\widetilde\ku\vphantom\ku^{t-s}(X).
  \end{equation*}
  Along the line contributing to \(\widetilde\ku\vphantom\ku^{0}(X)\),
  the only nonzero terms are
  \begin{align*}
    E_{1}^{2,2}&=H_{\Bet}^{4}(X;\ZZ(2))\simeq\ZZ/2,&
    E_{1}^{3,3}&=H_{\Bet}^{6}(X;\ZZ(3))\simeq\ZZ/2.
  \end{align*}
  The second term survives until~\(E_{\infty}\)
  for dimensional reasons.
  Therefore, \(\gr_{\AH}^{2}\ku^{0}(X)=E_{\infty}^{2,2}\) vanishes
  (in fact this dies on~\(E_{2}\) for dimensional reasons).
\end{example}

\begin{remark}\label{xvmuzo}
  One way to fix the failure for \(A=\Cls{C}(X)\)
  is simply to consider the usual motivic cohomology
  instead of the lisse variant.
  Indeed, in~\cite{mot-cx},
  we construct
  an isomorphism
  \(H_{\mot}^{2q}(\Cls{C}(X);\ZZ(q))\simeq H_{\Bet}^{2q}(X;\ZZ(q))\)
  for any~\(q\geq0\) and compactum~\(X\).

  However, this fix does not work in general.
  Let~\(P\) be a finite polyhedral model
  with the homotopy type of the space in \cref{counterexample},
  embedded semialgebraically in \(\RR^N\subset\CC^N\).
  Then~\(P\) is a semianalytic Stein compactum,
  and \(\Cls{O}(P)\) is regular
  (see, e.g.,~\cite[Theorem~6.6]{k-ros-3}), hence Ind-smooth over~\(\CC\)
  (and hence over~\(\QQ\))
  by Popescu’s theorem.
  Therefore, the motivic cohomology of~\(\Cls{O}(P)\)
  coincides with its lisse variant.
  Choose a polynomially convex compact neighborhood~\(K\) of~\(P\)
  that is a compact regular neighborhood,
  so that \(K\) is homotopy equivalent to~\(P\).
  There are restriction homomorphisms
  \begin{equation*}
    \Cls{A}(K)\to\Cls{O}(P)\to\Cls{C}(P).
  \end{equation*}
  The same codimension-filtration argument
  as in \cref{xxjrs2} then shows that the
  image of
  \(
    H^4_{\mot}(\Cls{O}(P);\ZZ(2))\to H^4_{\Bet}(P;\ZZ(2))
  \)
  factors through \(\gr^2_{\AH}\ku^0(P)=0\).
  Since \(H^4_{\Bet}(K;\ZZ(2))\to H^4_{\Bet}(P;\ZZ(2))\) is an isomorphism,
  functoriality of the restriction maps shows that
  \begin{equation*}
    H^4_{\mot}(\Cls{A}(K);\ZZ(2))\to H^4_{\Bet}(K;\ZZ(2))
  \end{equation*}
  is not surjective.
\end{remark}

\bibliographystyle{plain}
\let\SS\oldSS \let\top\oldtop  \newcommand{\yyyy}[1]{}

\end{document}